\documentclass[12pt,oneside]{amsart}

      \usepackage{amssymb}

      \theoremstyle{plain}
      \newtheorem{theorem}{Theorem}[section]
      \newtheorem{lemma}[theorem]{Lemma}
      
      \newtheorem{prop}{Proposition}
      \theoremstyle{definition}
      \newtheorem{definition}[theorem]{Definition}
      \theoremstyle{conj}
      
      \theoremstyle{remark}
      
      \makeatletter
      \let\runauthor\@author
      \let\runtitle\@title
      \makeatother
      
\begin{document}

   \author{Masoud Khalkhali, Ali Moatadelro, Sajad Sadeghi
\vspace{0.5cm}\\ Department of Mathematics,  University of Western Ontario, 
 London, Ontario, Canada
}
 \date{}


   \email{masoud@uwo.ca, amotadel@uwo.ca, ssadegh3@uwo.ca }
   

   \title{A Scalar Curvature Formula For the Noncommutative 3-Torus 
     }


   \begin{abstract}
     We compute the  scalar curvature of a  curved noncommutative 3-torus. 
To perturb the flat metric,  the standard volume form on the noncommutative 3-torus is 
conformally perturbed and the corresponding perturbed Laplacian is analyzed.
 Using Connes' pseudodifferential calculus for the noncommutative 3-torus, we explicitly compute 
 the first three terms of the small time heat kernel expansion for the perturbed Laplacian. 
The third term of the expansion gives a local formula for the scalar curvature.  Finally, we show that in the classical limit when the deformation parameters vanish,  
our  formula coincides with the formula for the commutative case.

   \end{abstract}

   \date{}


   \maketitle

   \section{Introduction}
  From the very   beginning of  noncommutative geometry 
in \cite{con},    noncommutative tori have proved to be an invaluable
model  to understand and test many aspects of  noncommutative geometry. Curvature, 
one of the most important geometric invariants, is among those aspects. Defining a suitable curvature concept  in the noncommutative
 setting  is an important problem  at the heart of noncommutative geometry. More precisely, we are interested in curvature invariants
 of noncommutative Riemannian manifolds. In contrast, it should be noted  that curvature of connections and the corresponding Chern-Weil theory in the noncommutative setting has already been defined in \cite{con}. 

In their pioneering work \cite{Conn},  Connes and Tretkoff 
(cf. also \cite{cohen} for a preliminary version) took a first step in this direction and proved a Gauss-Bonnet theorem  for a  curved  noncommutative 2-torus 
equipped with 
a conformally deformed metric. In fact, they gave a spectral definition of curvature and computed its
 trace. This result was extended in \cite{kh} to noncommutative tori 
 equipped with an arbitrary translation invariant
complex structure and conformal perurbation of its metric. The full computation  of curvature in these examples was done  independely and simultaneously  in  \cite{CM} and \cite{kh2}. This line of work has been followed up  and extended in different directions in many papers  \cite{bhumar, kh3, am,das,sw,fk,fka,farrr,far1, ml,  kh4, km}.

 The approach used in the aforementioned papers is based on the heat kernel techniques and Connes' pseudodifferential calculus on noncommutative tori.
     In this paper using  a similar technique  we will give a formula for the scalar curvature of a curved  noncommutative 
     3-torus. This would be the first odd dimensional case that has been studied among the noncommutative 
     tori. In \cite{Liu} a general pattern for the scalar curvature of even dimensional  noncommutatuve tori is found which in some sense repaets the two dimensional 
     case \cite{CM, kh2}.  A  similar question in the odd dimensional case needs a close 
     study of the three dimensional case first. The concept of curvature in the noncommutative setting  has also been studied   through an algebraic   approach and a noncommutative analogue of the Levi-Civita connection in  \cite{am, ros}.

 This paper is organized as follows. In Section 2, we recall some facts about the heat kernel expansion in the commutative case. 
 In Section 3, we recall basic facts about higher 
dimensional noncommutative tori and their flat geometry. Then 
 we perturb the standard 
volume form on this space conformally  and analyse the 
corresponding perturbed Laplacian. In Section 4, we recall the pseudodifferential 
calculus of \cite{con1} for $\mathbb{T}_\theta^3$. In Section 5, we review the derivation of  the 
small time heat kernel expansion for the perturbed Laplacian, using the pseudodifferential calculus. 
  Then we perform the computation of the scalar curvature for 
$\mathbb{T}_\theta^3$, and find explicit formulas for the local functions that describe 
the curvature in terms of the modular automorphism of the conformally perturbed volume 
form and derivatives of the Weyl factor.

We would like to thank Asghar Ghorbanpour for his helpful discussions and comments on the subject
of this paper.
  \section{Heat Kernel Expansion and Scalar Curvature}
To motivate the definition of scalar curvature in our noncommutative setting, let us first recall Gilkey's theorem on asymptotic expansion of heat kernels for  the special case of Laplacians. 
   Let $(M,g)$ be a closed, oriented  Riemannian manifold of dimension $n$, endowed with the metric $g$
   and let $\bigtriangleup$ denote  the scalar Laplacian   acting on $C^{\infty}(M)$, the algebra of smooth functions on $M$. If $C$ is a contour going counterclockwise around the nonnegative part of the real axis without touching it, then using the Cauchy integral formula
   $$
   e^{-t\bigtriangleup}=\dfrac{1}{2\pi i} \int_{C} e^{-t\lambda}(\bigtriangleup- \lambda)^{-1} d\lambda
   $$
   and approximating the operator $(\bigtriangleup- \lambda)^{-1}$ by a pseudodifferential operator $R(\lambda)$  one can find an asymptotic expansion for the smooth kernel $K(t,x,y)$ of $e^{-t\bigtriangleup}$ on the diagonal \cite{gil}. 
   
   More precisely,
using the formula for the symbol of the product of two pseudo differential operators one can inductively find an asymptotic expansion $ \overset{\infty}{\underset{j=0}{\sum }}r_j(x,\xi,\lambda)$ for the symbol of $R(\lambda)$ such that $r_j(x,\xi,\lambda)$ is a symbol of order $-2-j$ depending on the complex parameter $\lambda$, where $j \in \mathbb{N}\cup \left\lbrace0 \right\rbrace$, $x\in M$ and $\xi \in \mathbb{R}^n$. Then one can    
    see that for $t>0$, the operator $e^{-t\bigtriangleup}$ has a smooth kernel $K(t,x,y)$ and as $t\longrightarrow 0^{+}$, there exist  an asymptotic expansion
$$
K(t,x,x)\sim t^{-n/2}\sum_{m=0}^{\infty} a_{2m}(x) t^{m},
$$
where
\begin{equation} \label{sc}
a_{2m}(x)=\dfrac{1}{2\pi i} \iint_{C} e^{-\lambda} r_{2m}(x,\xi,\lambda)d\lambda d\xi.
\end{equation} 
It follows that  we have an asymptotic expansion for the heat trace
$$
\text{Tr}_{L^2}e^{-t\bigtriangleup}\sim t^{-n/2}\sum_{m=0}^{\infty} a_{2m} t^{m} \quad \quad  t\to 0, 
$$
where
 $$a_{2m}=\int _{M}a_{2m}(x) \,\text{dvol}(x).$$ 
 Moreover,  it is known that $a_{2}(x)$ is a constant multiple of the scalar curvature of $M$ at the point $x$, so that 
 $a_{2}$ is (a multiple of) the total scalar curvature \cite{gil}. In what follows we will explain how we exploit these facts to define the scalar curvature of the curved noncommutative 3-torus by analogy.
 \section{ Curved Noncommutative 3-tori}

 Let $\theta=(\theta_{k \ell})\in M_{3}(\mathbb{R})$ be a skew symmetric matrix. The universal unital C*-algebra generated by
  three unitary elements $u_1,u_2,u_3$ subject to the relations 
 $$
 u_ku_\ell=e^{2\pi i \theta_{k \ell}}u_\ell u_k,\,\,\,\,\,\,\,\,\, k,\ell=1,2,3
 $$
 is called the noncommutative 3-torus and is denoted by $A_{\theta}^3$. It has a positive faithful 
 normalized trace denoted by $\tau$. This C*-algebra is indeed a noncommutative deformation of $C(\mathbb{T}^3)$, the algebra of continuous functions on the 3-torus. 
 
 For $r=(r_1,r_2,r_3)\in \mathbb{Z}^3$ we set 
 $$u^r=\text{exp}(\pi i(r_1 \theta_{12}r_2+r_1 \theta_{13}r_3+r_2 \theta_{23}r_3))u_1^{r_1}u_2^{r_2}u_3^{r_3}.$$
There is an action $\alpha$ of the 3-torus $\mathbb{T}^3$  on $A_{\theta}^3$ which is defined by  
$$
\alpha_z(u^r) =z^r u^r
$$ 
where $z=(z_1,z_2,z_3)\in
 \mathbb{T}^3 $ and $z^r=z_1^{r_1}z_2^{r_2}z_3^{r_3}$.
Let $\mathbb{T}_{\theta}^3$ be the set of all elements $a\in A_{\theta}^3$ for which the map 
$$
\alpha(a): \mathbb{T}^3 \longrightarrow A_{\theta}^3,\,\,\,\,\,\,\, z\mapsto \alpha_z(a),
$$ 
is a smooth map. This set is a unital dense subalgebra of $A_{\theta}^3$ and it is called the algebra of smooth elements of $A_{\theta}^3$. In fact, it is the analogue of $C^{\infty}(\mathbb{T}^3)$, the algebra
 of smooth functions on the 3-torus. It is known that
 $$
\mathbb{T}_{\theta}^3=\left\lbrace
\sum _{r\in \mathbb{Z}^3} a_r u^r:   (a_r) \text{ is a rapidly decreasing function on  }\mathbb{Z}^3
\right\rbrace.
 $$
 By rapidly decreasing we mean for all $k\in \mathbb{N}$,
 $$
\text{Sup}(1+\vert r\vert ^2)^k \vert a_r\vert ^2< \infty. 
 $$
 
 The  trace $\tau$ on $A_{\theta}^3$  plays the role of integration in the noncommutative setting and extracts the constant term of the elements of $\mathbb{T}_{\theta}^3$, i.e.
$$
\tau(\sum _{r\in \mathbb{Z}^3} a_r u^r)=a_0.
$$
 The algebra $\mathbb{T}_{\theta}^3$ also possesses three derivations, uniquely defined by the relations 
 $$\delta_j(\sum _{r\in \mathbb{Z}^3} a_r u^r)=\sum _{r\in \mathbb{Z}^3} r_ja_r u^r,\,\,\,\,\,\,\,\,\,j=1,2,3.$$
 One can see that for $a\in \mathbb{T}_{\theta}^3$,
\begin{equation} \label{anco}
 (\delta_j(a))^*=-\delta_j(a^*).
\end{equation}

These derivations are noncommutative counterparts of the partial derivatives on $C^{\infty}(\mathbb{T}^3)$ and they satisfy the integration by parts relation i.e. 
\begin{equation} \label{INP}
\tau (a\delta_j(b))=-\tau (\delta_j(a)b), \,\,\,\,\,\,\,a,b \in \mathbb{T}_{\theta}^3.
\end{equation}

Our next goal is to introduce $ \bigtriangleup$,  the Laplace operator on $\mathbb{T}_{\theta}^3$. Then perturbing the metric in a conformal class we 
will define the perturbed Laplace operator $ \bigtriangleup_\varphi$. Let $k \in \mathbb{T}_{\theta}^3$ be a positive element
 representing the conformal class of the metric on $A_{\theta}^3$. We shall show that $\bigtriangleup_\varphi$ is anti-unitarily equivalent to the operator $P$ where
   $$P=
k\bigtriangleup k^{3}+\sum _{j=1}^{3}k^{3}\delta_j
(k^{-2})\delta_jk^{3}
$$
  and we will use the latter to define the scalar curvature of $A_{\theta}^3$ with a conformally perturbed metric. In fact, by analogy with (\ref{sc}), we define the scalar curvature of $A_{\theta}^3$ with the perturbed metric to be 
\begin{equation} \label{SCT}
\dfrac{k^6}{2\pi i} \int _{\mathbb{R}^3}\int_{C} e^{-\lambda} b_{2}(\xi,\lambda)d\lambda d\xi,
\end{equation}
where $C$ is a contour going counterclockwise around the nonnegative part of the real axis and $ b_{2}(\xi,\lambda)$ is the third term in the asymptotic expansion of the symbol of $(P-
     \lambda)^{-1}$. We will find the first three terms of this asymptotic expansion by   Connes' pseudodifferential calculus \cite{con} and finally we will compute (\ref{SCT}).
  
Let $ \left\langle .,. \right\rangle _\tau$ be the inner product on $A_{\theta}^3$ defined by
$$
\left\langle a,b \right\rangle_\tau=\tau(b^*a), \,\,\,\,\, a,b \in A_{\theta}^3 .
$$
We denote the completion of $A_{\theta}^3$ with respect to this inner product by $H_\tau$. It is indeed the representation 
space in the GNS construction associated to $\tau$. We define an unbounded operator
$$
d:  \mathbb{T}_{\theta}^3 \subset H_\tau \longrightarrow H_\tau \times H_\tau \times H_\tau
$$
by $d(a)=(\delta_1(a), \delta_2(a), \delta_3(a))$ for $a\in \mathbb{T}_{\theta}^3 $. This operator is defined using an 
analogy with the classic case. Indeed, for $\mathbb{T}^3$, the classic 3-torus, 
the de Rham operator is an operator from $C^{\infty}(\mathbb{T}^3)$ to $\Omega^{1}(\mathbb{T}^3)=C^{\infty}(\mathbb{T}^3) \otimes 
\mathbb{C}^3 $, where $C^{\infty}(\mathbb{T}^3)$ is the space of smooth functions 
and $\Omega^{1}(\mathbb{T}^3)$ is the space of 1-forms on $\mathbb{T}^3$. 

Let $x,y,z,b \in  \mathbb{T}_{\theta}^3$ and $\left\langle \cdot , \cdot \right\rangle_0$ be the inner product of $H_\tau \times H_\tau \times H_\tau$.
 We have, for the adjoint $d^*$,
$$\left\langle d^*(x,y,z),b \right\rangle_\tau =\left\langle (x,y,z),d(b) \right\rangle_0=\left\langle (x,y,z),(\delta_1(b),\delta_2(b),\delta_3(b)) \right\rangle_0 =$$
$$
-\tau(\delta_1(b^*)x)-\tau(\delta_2(b^*)y)-\tau(\delta_3(b^*)z).
$$
Using integration by parts (\ref{INP}), we obtain 
$$\left\langle d^*(x,y,z),b \right\rangle_\tau =
\tau(b^*\delta_1(x))+
\tau(b^*\delta_2(y))+
\tau(b^*\delta_3(z))=$$
$$
\tau(b^*(\delta_1(x)+\delta_2(y)+\delta_3(z)))=
\left\langle \delta_1(x)+\delta_2(y)+\delta_3(z),b \right\rangle_\tau.
$$
Therefore, $d^*(x,y,z)=\delta_1(x)+\delta_2(y)+\delta_3(z)$.
Now we define the Laplace operator $$ \bigtriangleup : \mathbb{T}_{\theta}^3 \subset H_\tau \longrightarrow H_\tau, $$ by $ \bigtriangleup =d^*d$. Note that $$\bigtriangleup(a)=d^*d(a)=d^*
(\delta_1(a), \delta_2(a), \delta_3(a))=\delta_1^2(a)+\delta_2^2(a)+\delta_3^2(a),$$
so $\bigtriangleup=\delta_1^2+\delta_2^2+\delta_3^2$.

Next we will conformally perturb the metric on $\mathbb{T}_{\theta}^3$. Let $h \in \mathbb{T}_{\theta}^3$ be a self adjoint element, we define
a positive linear functional $\varphi$ on $A_{\theta}^3$ by 
$$
\varphi(a)=\tau (ae^{-3h}),\,\,\,\,\, a\in A_{\theta}^3. 
$$  

Let $ \Delta$ be the modular operator for $\varphi$, i.e.
$$
\Delta(a)=e^{-3h}ae^{3h},\,\,\,\,\, a\in A_{\theta}^3,
$$
and $\left\lbrace\sigma_t\right\rbrace$, $t\in \mathbb{R}$ be a 1-parameter group of automorphisms of $A_{\theta}^3$ defined by
$$
\sigma_t(a)= \Delta^{-it}(a),\,\,\,\,\, a\in A_{\theta}^3.
$$
Unlike $\tau$, $\varphi$ is not a trace. But it satisfies the KMS condition at $\beta=1$ for $\left\lbrace\sigma_t\right\rbrace$. In other words, 
$$
\varphi(ab)=\varphi(b\sigma_i(a)),\,\,\,\,\, a,b\in A_{\theta}^3.
$$
Now we define an inner product on $A_{\theta}^3$ by
$$
\left\langle a,b \right\rangle_{\varphi }=\varphi (b^*a), \,\,\,\,\, a,b \in A_{\theta}^3 
$$
and we denote the Hilbert space completion of $A_{\theta}^3$ with this inner product by $H_\varphi$. 

To define the Laplace operator on $H_\varphi$, again we mimic the classic case. Let 
$g$ be the flat metric on the 3-torus $\mathbb{T}^3$ and $$\tilde{g}=e^{-2h} g.$$ Clearly $\tilde{g}^{-1} = e^{2h}g^{-1}$ and $\text{dvol}_{\tilde{g}}=e^{-3h}\text{dvol}_{g} $. Therefore, for $\alpha, \beta \in \Omega^{1}(\mathbb{T}^3) $ we have
$$
\left\langle \alpha, \beta \right\rangle_{\tilde{g}}=
\int \tilde{g}^{-1}(\alpha_p, \beta_p) \,\,\text{dvol}_{\tilde{g}}=\int e^{2h}g^{-1}(\alpha_p, \beta_p) e^{-3h} \,\,\text{dvol}_{g}
$$
$$=\int g^{-1}(\alpha_p, \beta_p) e^{-h} \,\,\text{dvol}_{g}.
$$

We also define
a positive linear functional $\psi$ on $A_{\theta}^3$ by 
$$
\psi(a)=\tau (ae^{-h}),\,\,\,\,\, a\in A_{\theta}^3,
$$  
and an inner product on $A_{\theta}^3$ by
$$
\left\langle a,b \right\rangle_{\psi }=\psi (b^*a), \,\,\,\,\, a,b \in A_{\theta}^3.
$$
We also denote the Hilbert space completion of $A_{\theta}^3$ with this inner product by $H_\psi$. 
Let $$
d_\varphi: \mathbb{T}_{\theta}^3 \subset  H_\varphi \longrightarrow H_\psi \times H_\psi \times H_\psi
$$
be defined by $d_\varphi(a)=(\delta_1(a), \delta_2(a), \delta_3(a))$ for $a\in \mathbb{T}_{\theta}^3 $. 

Let $x,y,z,b \in  \mathbb{T}_{\theta}^3$, $k=e^{h/2}$ and $\left\langle \cdot , \cdot \right\rangle_1$ be the inner product of $H_\psi \times H_\psi \times H_\psi$.
 We have 
$$\left\langle d_\varphi^*(x,y,z),b \right\rangle_\varphi =\left\langle (x,y,z),d_\varphi(b) \right\rangle_1=\left\langle (x,y,z),(\delta_1(b),\delta_2(b),\delta_3(b)) \right\rangle_1 =$$
$$
-\tau(\delta_1(b^*)xk^{-2})-\tau(\delta_2(b^*)yk^{-2})-\tau(\delta_3(b^*)zk^{-2}).
$$
Using integration by parts (\ref{INP}), we obtain 
$$\left\langle d_\varphi ^*(x,y,z),b \right\rangle_\varphi =
\tau(b^*\delta_1(xk^{-2}))+
\tau(b^*\delta_2(yk^{-2}))+
\tau(b^*\delta_3(zk^{-2}))=$$
$$
\tau(b^*(\delta_1(xk^{-2})+\delta_2(yk^{-2})+\delta_3(zk^{-2})))=$$
$$
\left\langle (\delta_1(xk^{-2})+\delta_2(yk^{-2})+\delta_3(zk^{-2}))k^6,b \right\rangle_\varphi.
$$
Therefore, $$d_\varphi ^*(x,y,z)=\delta_1(xk^{-2})k^6+\delta_2(yk^{-2})k^6+\delta_3(zk^{-2})k^6.$$
Now we define the perturbed Laplace operator $$ \bigtriangleup _\varphi :  \mathbb{T}_{\theta}^3 \subset H_\varphi \longrightarrow H_\varphi, $$ by $ \bigtriangleup _\varphi =d _\varphi^*d _\varphi$. Note that $$\bigtriangleup_\varphi(a)=d_\varphi^*d_\varphi(a)=d_\varphi^*
(\delta_1(a), \delta_2(a), \delta_3(a))=$$
$$
\delta_1(\delta_1(a)k^{-2})k^6+\delta_2(\delta_2(a)k^{-2})k^6+\delta_3(\delta_3(a)k^{-2})k^6
.$$
Since $$\delta_j(\delta_j(a)k^{-2}
)k^6=(\delta_j^2(a)k^{-2}+\delta_j(a)\delta_j(k^{-2}))k^6=\delta_j^2(a)k^4+\delta_j(a)\delta_j(k^{-2})k^6,$$
we have
$$\bigtriangleup_\varphi = \sum _{j=1}^{3}
R_{k^4}\delta_j^2+R_{(\delta_j
(k^{-2})k^6)}\delta_j,
$$
where for any element $x\in A_{\theta}^3$, by $R_x$ we mean the right multiplication operator by $x$. 

 Moreover, since
$$
\left\langle R_{k^{3}}a,R_{k^{3}}b \right\rangle_\varphi=\left\langle ak^{3},bk^{3} \right\rangle_\varphi=$$
$$\varphi(k^{3}b^*ak^{3})=\tau(k^{3}b^*ak^{3}k^{-6})=\tau(b^*a)=\left\langle a,b \right\rangle_\tau,
$$
$R_{k^{3}}$ extends to a unitary operator $W: H_\tau \longrightarrow H_\varphi$. Let 
$$J:H_\tau \longrightarrow H_\tau$$ be the anti unitary operator defined by 
$J(a)=a^*$. Then $$WJ:H_\tau \longrightarrow H_\varphi$$ is an anti unitalry operator and obviously $\bigtriangleup_\varphi$ is anti unitarily equivalent to 
\begin{equation} \label{dgh}
JW^*\bigtriangleup_\varphi WJ=JR_{k^{-3}}JJ\bigtriangleup_\varphi JJ R_{k^{3}}J.
\end{equation}
One can see that $$ JR_{k^{-3}}J=k^{-3},\quad  
 JR_{k^{3}}J=k^{3},\quad JR_{(\delta_j
(k^{-2})k^6)}J=-k^6\delta_j
(k^{-2}). $$ By $k^{-3}$, $k^{3}$ and $-k^6\delta_j
(k^{-2})$ we mean the left multiplication operator by these elements. Moreover, using  (\ref{anco}), we see that $J$ anticommutes with $\delta_j$. So
$$
J\bigtriangleup_\varphi J=\sum _{j=1}^{3}
JR_{k^4}\delta_j^2J+JR_{(\delta_j
(k^{-2})k^6)}\delta_jJ
$$
$$
=\sum _{j=1}^{3}
JR_{k^4}J\delta_j^2-JR_{(\delta_j
(k^{-2})k^6)}J\delta_j=\sum _{j=1}^{3}
k^4\delta_j^2+k^6\delta_j
(k^{-2})\delta_j.
$$
Using (\ref{dgh}), we see that  
$\bigtriangleup_\varphi$ is anti unitarily equivalent to 
$$
k^{-3} (\sum _{j=1}^{3}
k^4\delta_j^2+k^6\delta_j
(k^{-2})\delta_j)k^{3}=\sum _{j=1}^{3}
k\delta_j^2k^{3}+k^{3}\delta_j
(k^{-2})\delta_jk^{3}
$$

\section{Connes' Pseudodifferential Calculus}
In this section we will  recall Connes' pseudodifferential calculus that was introduced in \cite{con}.
We shall be primarily working in dimension three. 

For $n\in \mathbb{N} \cup \left\lbrace0\right\rbrace$, a \textit{differential operator} on $\mathbb{T}_{\theta}^3$ of order $n$ is a polynomial in $\delta_1,\delta_2,\delta_3$ of the form
$$
P(\delta_1,\delta_2,\delta_3)= \sum_{\vert j\vert \leqslant n} a_j \delta_1^{j_1}\delta_2^{j_2}\delta_3^{j_3}
$$
where $j=(j_1,j_2,j_3)\in \mathbb{Z}^3_{\geqslant0}$, $\vert j\vert=j_1+j_2+j_3$ and $a_j\in\mathbb{T}_{\theta}^3$. Now we extend this definition to pseudodifferential operators.

\begin{definition}
A smooth function $\rho: \mathbb{R}^3 \to \mathbb{T}_{\theta}^3$  is called a \textit{symbol of order} $n \in \mathbb{Z}$ if for all nonnegative integers $i_1,i_2,i_3,j_1,j_2,j_3$ there exists a constant $C$, such that 
 $$
 \Vert \delta_1^{i_1} \delta_2^{i_2} \delta_3^{i_3}( \partial_1^{j_1}\partial_2^{j_2}\partial_3^{j_3}\rho(\xi))\Vert
 \leq C(1+\vert \xi \vert)^{n-\vert j \vert}
, $$ 
and if there exists a smooth function  $ k: \mathbb{R}^3 \setminus \{(0,0,0)\} \to \mathbb{T}_{\theta}^3$
 such that 
$$
\lim_{\lambda \to \infty } \lambda^{-n}\rho(\lambda \xi_1,\lambda \xi_2,\lambda \xi_3)=k(\xi_1,\xi_2,\xi_3).
$$
\end{definition}
In the last definition by $\partial_1,\partial_2,\partial_3$ we mean partial derivatives, i.e.
$$
\partial_1=\partial/\partial \xi_1, \quad
\partial_2=\partial/\partial \xi_2, \quad
\partial_3=\partial/\partial \xi_3.
$$
The space of symbols of order $n$ is denoted by $S_n$.
To any symbol $\rho \in S_n$, an operator $P_\rho$ on $\mathbb{T}_{\theta}^3$ is associated which is given by 
$$
P_\rho(a)=(2\pi)^{-3}\iint e^{-iz.\xi}\rho (\xi) \alpha_z(a)dzd\xi,\,\,\,\,\,\,\,\ a \in \mathbb{T}_{\theta}^3,
$$
and is called a \textit{pseudodifferential operator}.  

\begin{definition}
Let $\rho$ and $\rho^{\prime}$ be symbols of order $k$. They are called \textit{equivalent} if and only if $\rho-\rho^{\prime} \in S_n$ for all $n\in\mathbb{Z}$. This equivalence relation is denoted by $\rho\sim \rho^{\prime}$.
\end{definition}
The next proposition which plays a key role in our computations in this paper, shows that the space of pseudodifferential operators is an algebra. Given the pseudodifferential operators $P$ and $Q$, by the next proposition we can find the symbols of $PQ$ and $P^{*}$ up to the equivalence relation $\sim$, where $P^{*}$ is the adjoint of $P$ with respect to the inner product $
\left\langle \cdot ,\cdot \right\rangle_\tau$ on $H_\tau$ (See \cite{Conn}).

\begin{prop} \label{pr}
Let $\rho$ and $\rho^{\prime}$ be the symbols of the pseudodifferential operators $P$ and $Q$. Then $PQ$ and $P^{*}$ are pseudodifferential operators, and $\sigma(PQ) $ and $\sigma(P^{*})$, symbols of $PQ$ and $P^{*}$ respectively, can be obtained by the following formulas
$$
\sigma(PQ)\sim \underset{(\ell_1,\ell_2,\ell_3) \in (\mathbb{Z}\geqslant 0)^3}{\sum }\dfrac{1}{\ell_1!\ell_2!\ell_3!} 
\partial_1^{\ell_1} \partial_2^{\ell_2}\partial_3^{\ell_3}
(\rho (\xi))
\delta_1^{\ell_1}
\delta_2^{\ell_2}\delta_3^{\ell_3}
(\rho^{\prime} (\xi)),
$$
$$
\sigma(P^*)\sim \underset{(\ell_1,\ell_2,\ell_3) \in (\mathbb{Z}\geqslant 0)^3}{\sum }\dfrac{1}{\ell_1!\ell_2!\ell_3!} 
\partial_1^{\ell_1} \partial_2^{\ell_2}\partial_3^{\ell_3}
\delta_1^{\ell_1}
\delta_2^{\ell_2}\delta_3^{\ell_3}
(\rho (\xi))^*.
$$
\end{prop}

\section{The Main Result}
In this section using Connes' pseudodifferential calculus we will define the scalar curvature of $A_{\theta}^3$ with a perturbed metric and we will compute it. 

Let $$P=
k\bigtriangleup k^{3}+\overset{3}{\underset{j=1}{\sum }}k^{3}\delta_j
(k^{-2})\delta_jk^{3}.
$$
As we mentioned in Section 3, $\bigtriangleup_\varphi$ is antiunitarily equivalent to the operator $P$
on $H_\tau$. So to study the spectral geometry of $A_{\theta}^3$ with a perturbed metric we work with the operator $P$. Exploiting
 the formula in Proposition \ref{pr},
  and considering $k, k^3, k^3\delta_j(k^{-2})$ as  pseudodifferential operators 
 of order $0$ with symbols $\sigma(k)=k,\sigma(k^3)=k^3,\sigma(k^3\delta_j(k^{-2}))=k^3\delta_j(k^{-2})$, plus the fact that the symbols of $\bigtriangleup$ and $ \overset{3}{\underset{j=1}{\sum }} \delta_j$ are $$
 \sigma(\bigtriangleup)= \overset{3}{\underset{i=1}{\sum }}\xi_i^2,\,\,\,\,\quad \sigma( \overset{3}{\underset{j=1}{\sum }} \delta_j)= \overset{3}{\underset{i=1}{\sum }}\xi_i$$ we can 
 find the symbol of $P$. Indeed, we can show that
$$
\sigma(P)= a_0(\xi)+a_1(\xi)+a_2(\xi),
$$
where $\xi=(\xi_1,\xi_2,\xi_3)$ and

\begin{eqnarray*}
a_0(\xi) &=&\overset{3}{\underset{i=1}{\sum }} (k\delta_i^2(k^3)+k^3\delta_i (k^{-2})\delta_i (k^{3 })),\\
a_1(\xi) &=&\overset{3}{\underset{i=1}{\sum }}\xi_i (2k\delta_i(k^3)+ k^3 \delta_i(k^{-2})k^3)\\
&=&\overset{3}{\underset{i=1}{\sum }}\xi_i (2k\delta_i(k^3)+k^3\delta_i(k) -k \delta_i(k^{3}))\\
&=&\overset{3}{\underset{i=1}{\sum }}\xi_i (k\delta_i(k^3)+k^3\delta_i(k) ),\\
a_2(\xi) &=&\overset{3}{\underset{i=1}{\sum }} k^4 \xi_i^2.
\end{eqnarray*}

Let $\lambda \in \mathbb{C}$. As mentioned in the introduction, to define the scalar curvature of $A_{\theta}^3$ with a conformally perturbed metric, we need to find an asymptotic expansion of the symbol of $(P-\lambda)^{-1}$. Indeed,  we have to find an operator $R_\lambda$ such that 
$$
\sigma(R_\lambda\cdot(P-\lambda))\sim \sigma(I)
$$ 
where $I$ is the identity operator. Using the formula in Proposition \ref{pr}, and following the steps in page 52 of \cite{gil}, we can find a recursive formula for the terms of an asymptotic expansion of $(P-\lambda)^{-1}$. In fact, one can show that 
$$\sigma(P-\lambda)^{-1}\sim \overset{\infty}{\underset{n=0}{\sum }} b_n(\xi,\lambda),$$
where $b_n(\xi,\lambda)$ is a symbol of order $-2-n$ given by the following recursive formula:
$$
b_0(\xi, \lambda)=(k^4\overset{3}{\underset{i=1}{\sum }}  \xi_i^2-\lambda)^{-1},
$$
\begin{equation}\label{rec}
b_n(\xi, \lambda)=-\overset{}{\underset{\underset{0\leq j< n,\,\,0\leq m\leq 2 }{2+j+\ell_1+\ell_2+\ell_3-m=n}}{\sum }}
\dfrac{1}{\ell_1!\ell_2!\ell_3!} 
\partial_1^{\ell_1} \partial_2^{\ell_2}\partial_3^{\ell_3}
(b_j)
\delta_1^{\ell_1}
\delta_2^{\ell_2}\delta_3^{\ell_3}
(a_m)b_0,
\end{equation}
for $n\geqslant1$.

Now we are able to define the scalar curvature of $A_{\theta}^3$ with a conformally perturbed metric. Indeed, using the notations that we have introduced, (\ref{sc}) motivates us to define the scalar curvature of $A_{\theta}^3$ with a conformally perturbed metric as follows:
\begin{definition}
Let $C$ be a  contour going counterclockwise around the nonnegative part of the real axis, and $ b_{2}(\xi,\lambda)$, for $\lambda \in \mathbb{C}$, be the third term in the asymptotic expansion of the symbol of $(P-
     \lambda)^{-1}$. Then the scalar curvature of $(A_{\theta}^3,\varphi)$ is defined to be the element $S \in A_{\theta}^3$ given by 
     $$
S=\dfrac{k^6}{2\pi i} \int _{\mathbb{R}^3}\int_{C} e^{-\lambda} b_{2}(\xi,\lambda)d\lambda d\xi.
$$
\end{definition}

Let $$\alpha(\lambda)=\int_{\mathbb{R}^3} b_2(\xi,\lambda) d\xi.$$
The function $\alpha$ is homogeneous of degree $-1/2$ with respect to $\lambda$. We also define $$\beta(\lambda)=\lambda^{-1/2}\alpha(\lambda).$$
The function $\beta$ is homogeneous of degree $-1$ with respect to $\lambda$. For the square root 
function we consider the nonnegative part of the real axis as the branch cut. Then we have
\[S=
\frac{k^6}{2\pi i}\beta(-1)\int_C \frac{e^{-\lambda}}{-\lambda^{1/2}} d\lambda.
\]
To compute the latter, we consider the contour $C=C_1+C_2+C_3$, where $C_1=re^{i\pi /4}$ for $ r\in(\infty,1)$, $C_2=e^{i\theta}$ for $ \theta\in(\pi/4 , 7\pi /4)$ and $C_3=re^{7 i\pi /4}$ for $ r\in(1, \infty)$. One can see that 
\[
\int_{C1}\frac{e^{-\lambda}}{-\lambda^{1/2}} d\lambda=(-1)^{7/8} e^{\frac{i \pi }{8}} \sqrt{\pi } \text{Erfc}\left[(-1)^{1/8}\right],
\]
\[\int_{C_2}\frac{e^{-\lambda}}{-\lambda^{1/2}} d\lambda=\sqrt{\pi } \left(-\text{Erf}\left[(-1)^{1/8}\right]+\text{Erf}\left[(-1)^{7/8}\right]\right),
\]
and
\[
\int_{C_3}\frac{e^{-\lambda}}{-\lambda^{1/2}} d\lambda=(-1)^{1/8} e^{\frac{7 i \pi }{8}} \sqrt{\pi } \left(1+\text{Erf}\left[(-1)^{7/8}\right]\right).
\]
Therefore,
\[ 
\int_C \frac{e^{-\lambda}}{-\lambda^{1/2}} d\lambda=-2\sqrt {\pi}
\]
and this implies that 
\[S=\frac{k^6}{2\pi i}\int_{\mathbb R^3}\int_C e^{-\lambda}b_2(\xi,\lambda) d\xi d\lambda=\frac{-k^6}{\sqrt \pi}\alpha(-1).
\]
By this argument to find $S$, it suffices to work with $\lambda=-1$ and compute
$$\alpha(-1)=\int_{\mathbb{R}^3} b_2(\xi,-1) d\xi.$$
We devote the rest of the paper to the calculation of $\alpha(-1)$.
\section{The Computation of  $b_2(\xi,-1)$ }
In this section we will use the recursive formula (\ref{rec}) to find $b_2(\xi,-1)$.
In what follows, we set $b_n=b_n(\xi,-1)$ for $n\in \mathbb{N}$. 

We know that $$b_0=(k^4\overset{3}{\underset{i=1}{\sum }}  \xi_i^2+1)^{-1}.$$
Now using (\ref{rec}), we have 
$$b_1=-b_0a_1b_0-(\sum_{i=1}^{3}\partial_{i}(b_{0}) \delta_{i}(a_{2}))b_0.
$$
Computing the above formula and using the result  in (\ref{rec}) we  obtain $b_2$. We have 
$$
b_{2}=-b_{0}a_{0}-b_{1}a_{1}
$$
$$-\partial_{1}(b_{0}) \delta_{1}(a_{1})- \partial_{2}(b_{0}) \delta_{2}(a_{1})-\partial_{3}(b_{0}) \delta_{3}(a_{1})
$$
$$
-\partial_{1}(b_{1}) \delta_{1}(a_{2})-\partial_{2}(b_{1}) \delta_{2}(a_{2})-
\partial_{3}(b_{1}) \delta_{3}(a_{2})
$$
$$
-\partial_{12}(b_{0}) \delta_{1}(\delta_{2}(a_{2}))- 
\partial_{13}(b_{0}) \delta_{1}(\delta_{3}(a_{2}))-
\partial_{23}(b_{0}) \delta_{2}(\delta_{3}(a_{2}))
$$
$$
-(1/2) \partial_{11}(b_{0}) \delta_{1}^{2}(a_{2})
-(1/2) \partial_{22}(b_{0}) \delta_{2}^{2}(a_{2})
-(1/2) \partial_{33}(b_{0}) \delta_{3}^{2}(a_{2}).
$$
After doing computations we get a simplified formula for $b_2$ which has more than 800 terms. In the next section we will use that simplified formula for $b_2$. 

\section{Integrating $b_2(\xi,-1)$ over $\mathbb{R}^3$  }
In this section, first we will change the variables and then we will use a rearrangement lemma to integrate $b_2(\xi,-1)$ over $\mathbb{R}^3$.

To integrate $b_2(\xi,-1)$ with respect to $\xi=(\xi_1,\xi_2,\xi_3)$, we apply the spherical change of coordinates 
\[
\xi _1= r \sin\Phi \cos\theta,\quad \xi _2= r\sin\Phi \sin\theta ,\quad\xi _3= r\cos\Phi, 
\]
where $0\leq\theta\leq 2\pi, 0\leq \Phi\leq \pi$ and $0\leq r\leq \infty$. Considering this change of coordinates and integrating with respect to $\theta$ and $\Phi$, one finds that
$$\int_{\mathbb{R}^3} b_2(\xi,-1) d\xi,$$ up to an overall factor of $4\pi/3$ is 
$$\int_{0}^{\infty} B(r) dr$$ where
$$B(r)=-6 r^2 b_0  k^2  \delta _j(k){}^2  b_0-3 r^2 b_0  k^3  \delta _j\left(\delta _j(k)\right)  b_0+12 r^4 b_0^2  k^6  \delta _j(k){}^2  b_0$$
$$+7 r^4 b_0^2  k^7  \delta _j\left(\delta _j(k)\right)  b_0-8 r^6 b_0^3  k^{10}  \delta _j(k){}^2  b_0-4 r^6 b_0^3  k^{11}  \delta _j\left(\delta _j(k)\right)  b_0$$
$$-6 r^2 b_0  k  \delta _j(k){}^2  b_0  k-3 r^2 b_0  k  \delta _j\left(\delta _j(k)\right)  b_0  k^2-3 r^2 b_0  k^2  \delta _j\left(\delta _j(k)\right)  b_0  k$$
$$+8 r^4 b_0^2  k^4  \delta _j(k){}^2  b_0  k^2+3 r^4 b_0^2  k^4  \delta _j\left(\delta _j(k)\right)  b_0  k^3+10 r^4 b_0^2  k^5  \delta _j(k){}^2  b_0  k$$
$$+5 r^4 b_0^2  k^5  \delta _j\left(\delta _j(k)\right)  b_0  k^2+5 r^4 b_0^2  k^6  \delta _j\left(\delta _j(k)\right)  b_0  k$$
$$-8 r^6 b_0^3  k^8  \delta _j(k){}^2  b_0  k^2-4 r^6 b_0^3  k^8  \delta _j\left(\delta _j(k)\right)  b_0  k^3$$
$$-8 r^6 b_0^3  k^9  \delta _j(k){}^2  b_0  k-4 r^6 b_0^3  k^9  \delta _j\left(\delta _j(k)\right)  b_0  k^2$$
$$-4 r^6 b_0^3  k^{10}  \delta _j\left(\delta _j(k)\right)  b_0  k-6 r^2 b_0  k  \delta _j(k)  k  \delta _j(k)  b_0$$
$$+10 r^4 b_0^2  k^4  \delta _j(k)  k^2  \delta _j(k)  b_0+12 r^4 b_0^2  k^5  \delta _j(k)  k  \delta _j(k)  b_0$$
$$-8 r^6 b_0^3  k^8  \delta _j(k)  k^2  \delta _j(k)  b_0-8 r^6 b_0^3  k^9  \delta _j(k)  k  \delta _j(k)  b_0$$
$$+5 r^4 b_0  k  \delta _j(k)  b_0  k^5  \delta _j(k)  b_0-2 r^6 b_0  k  \delta _j(k)  b_0^2  k^9  \delta _j(k)  b_0$$
$$+5 r^4 b_0  k^2  \delta _j(k)  b_0  k^4  \delta _j(k)  b_0-2 r^6 b_0  k^2  \delta _j(k)  b_0^2  k^8  \delta _j(k)  b_0$$
$$
+10 r^4 b_0  k^3  \delta _j(k)  b_0  k^3  \delta _j(k)  b_0+6 r^4 b_0  k^3  \delta _j(k)  b_0  \delta _j(k)  b_0  k^3$$
$$-4 r^6 b_0  k^3  \delta _j(k)  b_0^2  k^7  \delta _j(k)  b_0+8 r^4 b_0^2  k^4  \delta _j(k)  k  \delta _j(k)  b_0  k$$
$$-14 r^6 b_0^2  k^4  \delta _j(k)  b_0  k^6  \delta _j(k)  b_0+4 r^8 b_0^2  k^4  \delta _j(k)  b_0^2  k^{10}  \delta _j(k)  b_0$$
$$-16 r^6 b_0^2  k^5  \delta _j(k)  b_0  k^5  \delta _j(k)  b_0+4 r^8 b_0^2  k^5  \delta _j(k)  b_0^2  k^9  \delta _j(k)  b_0$$
$$-16 r^6 b_0^2  k^6  \delta _j(k)  b_0  k^4  \delta _j(k)  b_0+4 r^8 b_0^2  k^6  \delta _j(k)  b_0^2  k^8  \delta _j(k)  b_0$$
$$-18 r^6 b_0^2  k^7  \delta _j(k)  b_0  k^3  \delta _j(k)  b_0-14 r^6 b_0^2  k^7  \delta _j(k)  b_0  \delta _j(k)  b_0  k^3$$
$$
-8 r^6 b_0^3  k^8  \delta _j(k)  k  \delta _j(k)  b_0  k+8 r^8 b_0^3  k^8  \delta _j(k)  b_0  k^6  \delta _j(k)  b_0$$
$$+8 r^8 b_0^3  k^9  \delta _j(k)  b_0  k^5  \delta _j(k)  b_0+8 r^8 b_0^3  k^{10}  \delta _j(k)  b_0  k^4  \delta _j(k)  b_0$$
$$+8 r^8 b_0^3  k^{11}  \delta _j(k)  b_0  k^3  \delta _j(k)  b_0+8 r^8 b_0^3  k^{11}  \delta _j(k)  b_0  \delta _j(k)  b_0  k^3
$$
$$+3 r^4 b_0  k  \delta _j(k)  b_0  k^2  \delta _j(k)  b_0  k^3+4 r^4 b_0  k  \delta _j(k)  b_0  k^3  \delta _j(k)  b_0  k^2$$
$$+4 r^4 b_0  k  \delta _j(k)  b_0  k^4  \delta _j(k)  b_0  k-2 r^6 b_0  k  \delta _j(k)  b_0^2  k^6  \delta _j(k)  b_0  k^3$$
$$-2 r^6 b_0  k  \delta _j(k)  b_0^2  k^7  \delta _j(k)  b_0  k^2-2 r^6 b_0  k  \delta _j(k)  b_0^2  k^8  \delta _j(k)  b_0  k$$
$$+3 r^4 b_0  k^2  \delta _j(k)  b_0  k  \delta _j(k)  b_0  k^3+4 r^4 b_0  k^2  \delta _j(k)  b_0  k^2  \delta _j(k)  b_0  k^2$$
$$+4 r^4 b_0  k^2  \delta _j(k)  b_0  k^3  \delta _j(k)  b_0  k-2 r^6 b_0  k^2  \delta _j(k)  b_0^2  k^5  \delta _j(k)  b_0  k^3$$
$$-2 r^6 b_0  k^2  \delta _j(k)  b_0^2  k^6  \delta _j(k)  b_0  k^2-2 r^6 b_0  k^2  \delta _j(k)  b_0^2  k^7  \delta _j(k)  b_0  k$$
$$+8 r^4 b_0  k^3  \delta _j(k)  b_0  k  \delta _j(k)  b_0  k^2+8 r^4 b_0  k^3  \delta _j(k)  b_0  k^2  \delta _j(k)  b_0  k$$
$$-4 r^6 b_0  k^3  \delta _j(k)  b_0^2  k^4  \delta _j(k)  b_0  k^3-4 r^6 b_0  k^3  \delta _j(k)  b_0^2  k^5  \delta _j(k)  b_0  k^2$$
$$+4 r^8 b_0^2  k^7  \delta _j(k)  b_0^2  k^7  \delta _j(k)  b_0-4 r^6 b_0  k^3  \delta _j(k)  b_0^2  k^6  \delta _j(k)  b_0  k
$$
$$-10 r^6 b_0^2  k^4  \delta _j(k)  b_0  k^3  \delta _j(k)  b_0  k^3-12 r^6 b_0^2  k^4  \delta _j(k)  b_0  k^4  \delta _j(k)  b_0  k^2$$
$$-12 r^6 b_0^2  k^4  \delta _j(k)  b_0  k^5  \delta _j(k)  b_0  k+4 r^8 b_0^2  k^4  \delta _j(k)  b_0^2  k^7  \delta _j(k)  b_0  k^3$$
$$+4 r^8 b_0^2  k^4  \delta _j(k)  b_0^2  k^8  \delta _j(k)  b_0  k^2+4 r^8 b_0^2  k^4  \delta _j(k)  b_0^2  k^9  \delta _j(k)  b_0  k$$
$$-12 r^6 b_0^2  k^5  \delta _j(k)  b_0  k^2  \delta _j(k)  b_0  k^3-14 r^6 b_0^2  k^5  \delta _j(k)  b_0  k^3  \delta _j(k)  b_0  k^2$$
$$-14 r^6 b_0^2  k^5  \delta _j(k)  b_0  k^4  \delta _j(k)  b_0  k+4 r^8 b_0^2  k^5  \delta _j(k)  b_0^2  k^6  \delta _j(k)  b_0  k^3$$
$$+4 r^8 b_0^2  k^5  \delta _j(k)  b_0^2  k^7  \delta _j(k)  b_0  k^2+4 r^8 b_0^2  k^5  \delta _j(k)  b_0^2  k^8  \delta _j(k)  b_0  k$$
$$-12 r^6 b_0^2  k^6  \delta _j(k)  b_0  k  \delta _j(k)  b_0  k^3-14 r^6 b_0^2  k^6  \delta _j(k)  b_0  k^2  \delta _j(k)  b_0  k^2$$
$$-14 r^6 b_0^2  k^6  \delta _j(k)  b_0  k^3  \delta _j(k)  b_0  k+4 r^8 b_0^2  k^6  \delta _j(k)  b_0^2  k^5  \delta _j(k)  b_0  k^3$$
$$+4 r^8 b_0^2  k^6  \delta _j(k)  b_0^2  k^6  \delta _j(k)  b_0  k^2+4 r^8 b_0^2  k^6  \delta _j(k)  b_0^2  k^7  \delta _j(k)  b_0  k$$
$$+4 b_0^2  k^7  \delta _j(k)  b_0^2  k^4  \delta _j(k)  b_0  k^3 r^8+4 b_0^2  k^7  \delta _j(k)  b_0^2  k^5  \delta _j(k)  b_0  k^2 r^8$$
$$+4 b_0^2  k^7  \delta _j(k)  b_0^2  k^6  \delta _j(k)  b_0  k r^8+8 b_0^3  k^8  \delta _j(k)  b_0  k^3  \delta _j(k)  b_0  k^3 r^8$$$$+8 b_0^3  k^8  \delta _j(k)  b_0  k^4  \delta _j(k)  b_0  k^2 r^8+8 b_0^3  k^8  \delta _j(k)  b_0  k^5  \delta _j(k)  b_0  k r^8$$
$$+8 b_0^3  k^9  \delta _j(k)  b_0  k^2  \delta _j(k)  b_0  k^3 r^8+8 b_0^3  k^9  \delta _j(k)  b_0  k^3  \delta _j(k)  b_0  k^2 r^8 $$
$$+8 b_0^3  k^9  \delta _j(k)  b_0  k^4  \delta _j(k)  b_0  k r^8+8 b_0^3  k^{10}  \delta _j(k)  b_0  k  \delta _j(k)  b_0  k^3 r^8$$$$+8 b_0^3  k^{10}  \delta _j(k)  b_0  k^2  \delta _j(k)  b_0  k^2 r^8+8 b_0^3  k^{10}  \delta _j(k)  b_0  k^3  \delta _j(k)  b_0  k r^8$$
$$+8 b_0^3  k^{11}  \delta _j(k)  b_0  k  \delta _j(k)  b_0  k^2 r^8+8 b_0^3  k^{11}  \delta _j(k)  b_0  k^2  \delta _j(k)  b_0  k r^8$$
$$-16 b_0^2  k^7  \delta _j(k)  b_0  k  \delta _j(k)  b_0  k^2 r^6-16 b_0^2  k^7  \delta _j(k)  b_0  k^2  \delta _j(k)  b_0  k r^6$$
$$-3 r^2 b_0  k^2  \delta _j(k^{-1})  k^2  \delta _j(k)  b_0-3 r^2 b_0  k^2 \delta _j(k^{-1})  \delta _j(k)  b_0  k^2$$
$$
-3 r^2 b_0  k^3  \delta _j(k^{-1})    k  \delta _j(k)  b_0-3 r^2 b_0  k^3  \delta _j(k^{-1})  k^{-1}  \delta _j(k)  b_0  k^2
$$
$$-3 r^2 b_0  k^2  \delta _j(k^{-1})  k  \delta _j(k)  b_0  k-3 r^2 b_0  k^3  \delta _j(k^{-1})   \delta _j(k)  b_0  k$$
$$+4 r^8 b_0^2  k^7  \delta _j(k)  b_0^2  k^7  \delta _j(k)  b_0$$
In the above sum $b_0=(k^4r^2+1)^{-1}$. One can see that for $x \in \mathbb{T}_{\theta}^3$, $xk^n=k^n\Delta^{n/6}(x)$. Using this relation plus the fact that $kb_0=b_0k$ we can see that 
$$B(r)=-6 r^2 k^2  b_0  \delta _j(k){}^2  b_0-3 r^2 k^3b_0    \delta _j\left(\delta _j(k)\right)  b_0+12 r^4 k^6 b_0^2   \delta _j(k){}^2  b_0$$
$$+7 r^4 k^7b_0^2    \delta _j\left(\delta _j(k)\right)  b_0-8 r^6 k^{10}b_0^3    \delta _j(k){}^2  b_0-4 r^6 k^{11}b_0^3    \delta _j\left(\delta _j(k)\right)  b_0$$
$$-6k^2 r^2 b_0  \Delta^{1/6} ( \delta _j(k){}^2)  b_0  -3 r^2k^3 b_0   \Delta^{1/3} (\delta _j\left(\delta _j(k)\right) ) b_0 $$
$$-3 r^2k^3 b_0   \Delta^{1/6} (  \delta _j\left(\delta _j(k)\right) ) b_0 +8 r^4k^6 b_0^2  \Delta^{1/3} ( \delta _j(k){}^2)  b_0$$
$$  +3 r^4k^7 b_0^2    \Delta^{1/2} (\delta _j\left(\delta _j(k)\right)) b_0  +10 r^4 k^6b_0^2   \Delta^{1/6} (\delta _j(k){}^2 ) b_0  $$
$$+5 r^4k^7 b_0^2  \Delta^{1/3} (  \delta _j\left(\delta _j(k)\right)) b_0  +5 r^4k^7 b_0^2   \Delta^{1/6} (\delta _j\left(\delta _j(k)\right) ) b_0  $$
$$-8 r^6k^{10} b_0^3   \Delta^{1/3} (\delta _j(k){}^2 ) b_0  -4 r^6k^{11} b_0^3    \Delta^{1/2} (\delta _j\left(\delta _j(k)\right))  b_0  $$
$$-8 r^6 k^{10}b_0^3    \Delta^{1/6} (\delta _j(k){}^2)  b_0  -4 r^6 k^{11}b_0^3   \Delta^{1/3} (\delta _j\left(\delta _j(k)\right))  b_0  $$
$$-4 r^6k^{11} b_0^3    \Delta^{1/6} (\delta _j\left(\delta _j(k)\right) ) b_0  -6 r^2k^2 b_0    \Delta^{1/6} (\delta _j(k))    \delta _j(k)  b_0$$
$$+10 r^4 k^6 b_0^2 \Delta^{1/3} ( \delta _j(k)) \delta _j(k)  b_0+12 r^4k^6 b_0^2   \Delta^{1/6} (\delta _j(k))  \delta _j(k)  b_0$$
$$-8 r^6k^{10} b_0^3  \Delta^{1/3} ( \delta _j(k))   \delta _j(k)  b_0-8 r^6 k^{10} b_0^3   \Delta^{1/6} (\delta _j(k))    \delta _j(k)  b_0$$
$$+5 r^4 k^6 b_0  \Delta^{5/6} ( \delta _j(k))  b_0    \delta _j(k)  b_0-2 r^6 k^{10}b_0  \Delta^{3/2} (\delta _j(k))  b_0^2    \delta _j(k)  b_0$$
$$+5 r^4 k^6 b_0    \Delta^{2/3} (\delta _j(k)) b_0    \delta _j(k)  b_0-2 r^6k^{10} b_0  \Delta^{4/3} ( \delta _j(k))  b_0^2    \delta _j(k)  b_0$$
$$
+10 r^4k^6 b_0   \Delta^{1/2} ( \delta _j(k) ) b_0  \delta _j(k)  b_0+6 r^4k^6 b_0   \Delta^{1/2} (\delta _j(k))  b_0 \Delta^{1/2} ( \delta _j(k))  b_0  $$
$$+8 r^4 k^6 b_0^2  \Delta^{1/3} (\delta _j(k))    \Delta^{1/6} (\delta _j(k))  b_0  -14 r^6 k^{10} b_0^2    \Delta^{} (\delta _j(k))  b_0   \delta _j(k)  b_0$$
$$-4 r^6 k^{10} b_0  \Delta^{7/6} (\delta _j(k) ) b_0^2   \delta _j(k)  b_0+4 r^8k^{14} b_0^2 \Delta^{5/3} ( \delta _j(k))  b_0^2    \delta _j(k)  b_0$$
$$-16 r^6 k^{10}b_0^2   \Delta^{5/6} (\delta _j(k)) b_0   \delta _j(k)  b_0+4 r^8k^{14} b_0^2    \Delta^{3/2} (\delta _j(k))  b_0^2    \delta _j(k)  b_0$$
$$-16 r^6k^{10} b_0^2    \Delta^{2/3} (\delta _j(k))  b_0    \delta _j(k)  b_0+4 r^8 k^{14}b_0^2  \Delta^{4/3} ( \delta _j(k))  b_0^2   \delta _j(k)  b_0$$
$$-18 r^6 k^{10}b_0^2   \Delta^{1/2} (\delta _j(k))  b_0   \delta _j(k)  b_0-14 r^6 k^{10} b_0^2  \Delta^{1/2} ( \delta _j(k))  b_0  \Delta^{1/2} (\delta _j(k))  b_0  $$
$$
-8 r^6k^{10} b_0^3    \Delta^{1/3} (\delta _j(k)) \Delta^{1/6} (\delta _j(k))  b_0  +8 r^8 k^{14} b_0^3   \Delta^{} (\delta _j(k) ) b_0   \delta _j(k)  b_0$$
$$+8 r^8 k^{14}b_0^3   \Delta^{5/6} (\delta _j(k))  b_0   \delta _j(k)  b_0+8 r^8 k^{14} b_0^3  \Delta^{2/3} (\delta _j(k))  b_0    \delta _j(k)  b_0$$
$$+8 r^8 k^{14}b_0^3  \Delta^{1/2} (\delta _j(k))  b_0    \delta _j(k)  b_0+8 r^8k^{14} b_0^3   \Delta^{1/2} (\delta _j(k))  b_0  \Delta^{1/2} (\delta _j(k) ) b_0 
$$
$$+3 r^4k^6 b_0   \Delta^{5/6} ( \delta _j(k))  b_0   \Delta^{1/2} (\delta _j(k))  b_0  +4 r^4k^6 b_0    \Delta^{5/6} (\delta _j(k))  b_0  \Delta^{1/3} (\delta _j(k))  b_0  $$
$$+4 r^4k^6 b_0   \Delta^{5/6} (\delta _j(k))  b_0  \Delta^{1/6} ( \delta _j(k))  b_0  -2 r^6k^{10} b_0  \Delta^{3/2} (\delta _j(k))  b_0^2   \Delta^{1/2} (\delta _j(k))  b_0  $$
$$-2 r^6k^{10} b_0  \Delta^{3/2} (\delta _j(k))  b_0^2  \Delta^{1/3} (\delta _j(k))  b_0 -2 r^6 k^{10} b_0  \Delta^{3/2} (\delta _j(k))  b_0^2    \Delta^{1/6} (\delta _j(k))  b_0  $$
$$+3 r^4k^6 b_0  \Delta^{2/3} (\delta _j(k))  b_0  \Delta^{1/2} (\delta _j(k) ) b_0 +4 r^4 k^6b_0  \Delta^{2/3} (\delta _j(k) ) b_0  \Delta^{1/3} (\delta _j(k))  b_0 $$
$$+4 r^4k^6 b_0  \Delta^{2/3} (\delta _j(k) ) b_0   \Delta^{1/6} (\delta _j(k))  b_0 -2 r^6k^{10} b_0   \Delta^{4/3} (\delta _j(k))  b_0^2  \Delta^{1/2} (\delta _j(k))  b_0 $$
$$-2 r^6k^{10} b_0  \Delta^{4/3} (\delta _j(k))  b_0^2  \Delta^{1/3} (\delta _j(k))  b_0 -2 r^6 k^{10} b_0  \Delta^{4/3} (\delta _j(k))  b_0^2  \Delta^{1/6} (\delta _j(k))  b_0$$
$$+8 r^4k^6 b_0 \Delta^{1/2} (\delta _j(k))  b_0   \Delta^{1/3} (\delta _j(k))  b_0 +8 r^4 k^6 b_0  \Delta^{1/2} ( \delta _j(k))  b_0  \Delta^{1/6} (\delta _j(k))  b_0  $$
$$-4 r^6 k^{10} b_0  \Delta^{7/6} (\delta _j(k))  b_0^2   \Delta^{1/2} (\delta _j(k))  b_0 -4 r^6 k^{10}b_0  \Delta^{7/6} (\delta _j(k))  b_0^2   \Delta^{1/3} (\delta _j(k))  b_0 $$
$$+4 r^8 k^{14}b_0^2    \Delta^{7/6} (\delta _j(k))  b_0^2   \delta _j(k)  b_0-4 r^6 k^{10}b_0  \Delta^{7/6} ( \delta _j(k) ) b_0^2  \Delta^{1/6} (\delta _j(k))  b_0 
$$
$$-10 r^6k^{10} b_0^2  \Delta^{} ( \delta _j(k))  b_0  \Delta^{1/2} (\delta _j(k))  b_0 -12 r^6k^{10} b_0^2   \Delta^{} (\delta _j(k))  b_0   \Delta^{1/3} (\delta _j(k))  b_0 $$
$$-12 r^6 k^{10}b_0^2   \Delta^{} (\delta _j(k))  b_0  \Delta^{1/6} ( \delta _j(k))  b_0 +4 r^8k^{14} b_0^2   \Delta^{5/3} (\delta _j(k))  b_0^2    \Delta^{1/2} (\delta _j(k))  b_0 $$
$$+4 r^8k^{14} b_0^2  \Delta^{5/3} (\delta _j(k))  b_0^2  \Delta^{1/3} (\delta _j(k))  b_0+4 r^8k^{14} b_0^2   \Delta^{5/3} (\delta _j(k))  b_0^2  \Delta^{1/6} (\delta _j(k))  b_0 $$
$$-12 r^6 k^{10}b_0^2  \Delta^{5/6} ( \delta _j(k))  b_0  \Delta^{1/2} (\delta _j(k))  b_0 -14 r^6k^{10} b_0^2   \Delta^{5/6} (\delta _j(k) ) b_0  \Delta^{1/3} (\delta _j(k))  b_0 $$
$$-14 r^6 k^{10}b_0^2  \Delta^{5/6} (\delta _j(k))  b_0 \Delta^{1/6} (\delta _j(k))  b_0 +4 r^8k^{14} b_0^2    \Delta^{3/2} (\delta _j(k))  b_0^2    \Delta^{1/2} (\delta _j(k))  b_0  $$
$$+4 r^8k^{14} b_0^2   \Delta^{3/2} (\delta _j(k))  b_0^2  \Delta^{1/3} (\delta _j(k))  b_0 +4 r^8 k^{14} b_0^2  \Delta^{3/2} ( \delta _j(k))  b_0^2   \Delta^{1/6} (\delta _j(k))  b_0  $$
$$-12 r^6k^{10} b_0^2   \Delta^{2/3} ( \delta _j(k) ) b_0   \Delta^{1/2} (\delta _j(k) ) b_0 -14 r^6 k^{10}b_0^2   \Delta^{2/3} ( \delta _j(k))  b_0  \Delta^{1/3} (\delta _j(k))  b_0 $$
$$-14 r^6 k^{10}b_0^2  \Delta^{2/3} (\delta _j(k) ) b_0  \Delta^{1/6} (\delta _j(k))  b_0  +4 r^8k^{14} b_0^2    \Delta^{4/3} (\delta _j(k))  b_0^2  \Delta^{1/2} ( \delta _j(k))  b_0 $$
$$+4 r^8 k^{14}b_0^2   \Delta^{4/3} (\delta _j(k))  b_0^2  \Delta^{1/3} (\delta _j(k))  b_0 +4 r^8k^{14} b_0^2   \Delta^{4/3} (\delta _j(k))  b_0^2    \Delta^{1/6} (\delta _j(k)) b_0  $$
$$+4r^8 k^{14} b_0^2  \Delta^{7/6}( \delta _j(k))  b_0^2  \Delta^{1/2}(\delta _j(k))  b_0 +4 r^8 k^{14}b_0^2   \Delta^{7/6}(\delta _j(k))  b_0^2   \Delta^{1/3}(\delta _j(k))  b_0  $$
$$+4r^8 k^{14}b_0^2  \Delta^{7/6}(\delta _j(k))  b_0^2  \Delta^{1/6}( \delta _j(k))  b_0   +8 r^8 k^{14}b_0^3    \Delta^{}(\delta _j(k))  b_0   \Delta^{1/2}(\delta _j(k))  b_0   $$
$$+8r^8 k^{14} b_0^3  \Delta^{}(\delta _j(k))  b_0   \Delta^{1/3}(\delta _j(k))  b_0  +8 r^8 k^{14}b_0^3  \Delta^{}(\delta _j(k))  b_0   \Delta^{1/6}(\delta _j(k))  b_0  $$
$$+8r^8k^{14} b_0^3    \Delta^{5/6}(\delta _j(k))  b_0   \Delta^{1/2}( \delta _j(k)) b_0  +8r^8k^{14} b_0^3    \Delta^{5/6}(\delta _j(k))  b_0 \Delta^{1/3}(\delta _j(k))  b_0  $$
$$+8 r^8k^{14}b_0^3  \Delta^{5/6}( \delta _j(k) ) b_0   \Delta^{1/6}( \delta _j(k) ) b_0  +8r^8k^{14} b_0^3  \Delta^{2/3}( \delta _j(k))  b_0   \Delta^{1/2}(\delta _j(k))  b_0 $$
$$+8r^8k^{14} b_0^3  \Delta^{2/3}(\delta _j(k))  b_0    \Delta^{1/3}(\delta _j(k))  b_0  +8r^8 k^{14}b_0^3   \Delta^{2/3}( \delta _j(k))  b_0   \Delta^{1/6}( \delta _j(k))  b_0 $$
$$+8 r^8 k^{14}b_0^3  \Delta^{1/2}(\delta _j(k) ) b_0   \Delta^{1/3}(\delta _j(k))  b_0 +8r^8k^{14} b_0^3   \Delta^{1/2}( \delta _j(k))  b_0  \Delta^{1/6}( \delta _j(k))  b_0  $$
$$-16r^6 k^{10}b_0^2  \Delta^{1/2}( \delta _j(k) ) b_0    \Delta^{1/3}(\delta _j(k))  b_0   -16r^6 k^{10}b_0^2   \Delta^{1/2}(\delta _j(k))  b_0   \Delta^{1/6}(\delta _j(k))  b_0  $$
$$-3 r^2k^4 b_0  \Delta^{1/3}( \delta _j(k^{-1}))   \delta _j(k)  b_0-3 r^2 k^4b_0   \Delta^{1/3}(\delta _j(k^{-1}))  \Delta^{1/3}(\delta _j(k) ) b_0 $$
$$
-3 r^2k^4 b_0   \Delta^{1/6}( \delta _j(k^{-1}) )     \delta _j(k)  b_0-3 r^2 k^4b_0   \Delta^{1/6}(\delta _j(k^{-1})) \Delta^{1/3}( \delta _j(k))  b_0  
$$
$$-3 r^2k^4 b_0  \Delta^{1/3}(\delta _j(k^{-1}))  \Delta^{1/6}( \delta _j(k) ) b_0 -3 r^2k^4 b_0  \Delta^{1/6}(\delta _j(k^{-1}) ) \Delta^{1/6}(\delta _j(k))  b_0. $$
In the above formula, summation over $j=1,2,3$ is understood. 

To integrate the terms of $B(r)$, we need a lemma similar to the rearrangement lemma in \cite{CM} or its generalization in \cite{lesch}. In the following lemma we will use exactly the same method of the proof of the rearrangement lemma in \cite{CM}, to prove a slightly different statement. This is needed since the Jacobian of the spherical change of coordinates involves $r^2$ while in \cite{CM}, the polar coordinates are used and in that case the Jacobian involves $r$.
 
\begin{lemma} 
\label{RE} 
Let $\rho_j \in \mathbb{T}_{\theta}^3$ and $m_j \in \mathbb{Z}, 
$ for $j=0, 1, 2, \ldots, l$. 
Then 
$$
\int_{0}^{\infty} (k^4u+1)^{-m_0} \rho_1 (k^4u+1)^{-m_1} \cdots \rho_l(k^4u+1)^{-m_l} u^{(\sum _{j=0}^l m_j-3/2)} du
$$
$$
=k^{(-4\sum _{j=0}^lm_j+2 )} F_{m_0,m_1,m_2, \ldots, m_l}(\Delta_{(1)}, \Delta_{(2)}, \ldots, \Delta_{(l)})(\rho_1 \rho_2 \cdots \rho_l),
$$
where 
$$F_{m_0,m_1,m_2, \ldots, m_l}(u_1, u_2, \ldots, u_l)=$$
$$\int_{0}^{\infty}(u+1)^{-m_0} \overset{l}{\underset{j=1} {\prod}}\left ( u\overset{j}{\underset{h=1} {\prod}}u_h+1 \right )^{-m_j}u^{(\sum _{j=0}^l m_j-3/2)} du,$$
and $\Delta_{(j)}$ means that $\Delta$ acts on the $j$th factor, for $j=0, 1, 2, \ldots, l$. 
\end{lemma}
 
\begin{proof}
Let $G_n$ and $
G_{n,\alpha}
$ be the inverse Fourier transforms of 
the functions defined respectively by
$$g_n(t)=(e^{t/2}+e^{-t/2})^{-n}$$
and $$
H_{n,\alpha}(t)=e^{(n-\alpha)t}(e^{t}+1)^{-n},
$$
where $n \in \mathbb{N}$ and $\alpha \in (0,n)$.
Then $G_{n,\alpha}(s)=G_n(s-i(n/2-\alpha)) $. 
So we have  
\begin{equation} \label{h}
H_{n,\alpha}(t)= \int _{-\infty}^{\infty}G_n(s-i(n/2-\alpha)) e^{-ist}ds.
\end{equation}

Let $J$ be the integral in the left hand side of the equation in the lemma. Now we use the substitutions $u=e^s$ and $k=e^{f/4}$ to compute $J$. Therefore, we have
$$J=$$
$$\int_{-\infty}^{\infty}(e^{(s+f)}+1)^{-m_0} \rho_1 (e^{(s+f)}+1)^{-m_1} \cdots \rho_l(e^{(s+f)}+1)^{-m_l}e^{(\sum _{j=0}^lm_j-1)s} e^{s/2}ds$$
Then for $j=0, 1, 2, \ldots, l$, we pick a positive real number $\alpha_j$ such that $\sum _{j=0}^l \alpha_j=1$. We also set
 $\beta_j=-\sum _{i=j}^l (m_i-\alpha_i)$. Replacing $(e^{(s+f)}+1)^{-m_j}$ by $e^{(m_j-\alpha_j)(f+s)}(e^{(s+f)}+1)^{-m_j}$
in $J$, we get
$$J=e^{-(\sum _{j=0}^lm_j-1)f}$$
$$ \times \int_{-\infty}^{\infty}H_{m_0,\alpha_0}(s+f)\Delta^{\beta_1}(\rho_1)H_{m_1,\alpha_1}(s+f) \cdots \Delta^{\beta_l}(\rho_l)H_{m_l,\alpha_l}(s+f) e^{s/2}ds.$$
Let $\rho_j^{\prime}=\Delta^{\beta_j}(\rho_j)$. Using (\ref{h}), $J$ can be written as an integral of the form 
\begin{equation} \label{w}
e^{-(\sum _{j=0}^l m_j-1)f} H_{m_0,\alpha_0}(s+f)
\rho^\prime_1 e^{-i(s+f)t_1}\rho^\prime_2 \cdots e^{-i(s+f)t_{l-1}}\rho^\prime_l e^{-i(s+f)t_l} e^{s/2}
\end{equation}
with respect to the measure 
$
\overset{j=l}{\underset{j=1} {\prod}} G_{m_j,\alpha_j}(t_j)dt_jds.
$

Now we can write (\ref{w}) as
$$
e^{-(\sum _{j=0}^lm_j-1)f} H_{m_0,\alpha_0}(s+f) e^{-i(\sum _{j=1}^l t_j)(s+f)}\,\overset{l}{\underset{h=1} {\prod}} \Delta^{-i\sum _{j=h}^lt_j}(\rho ^\prime_h)e^{s/2}.
$$
We also have
$$
\int_{-\infty}^{\infty} H_{m_0,\alpha_0}(s+f) e^{-i(\sum _{j=1}^lt_j)(s+f)}e^{s/2} ds
=$$
$$e^{-f/2}\int_{-\infty}^{\infty}e^{(s+f)/2} H_{m_0,\alpha_0}(s+f) e^{-i(\sum _{j=1}^lt_j)(s+f)} ds=
2\pi e^{-f/2} P_{m_0,\alpha_0}(-\sum _{j=1}^lt_j),
$$
where $P_{m_0,\alpha_0}$ is the inverse Fourier transform of the function $e^{s/2}H_{m_0,\alpha_0}(s)$.

So we have $$
J=2\pi e^{-f/2} e^{-(\sum _{j=0}^lm_j-1)f}
\int 
\overset{h=l}{\underset{h=1} {\prod}} \Delta^{-i\sum _{j=h}^l t_j}(\rho ^\prime_h)
P_{m_0,\alpha_0}(-\overset{l}{\underset{j=1} {\sum}}t_j)
\overset{l}{\underset{j=1} {\prod}} G_{m_j,\alpha_j}(t_j)dt_j.
$$
Replacing $\rho ^\prime_j$ by $\Delta^{\beta_j}(\rho_j)$ we have
$$\Delta^{-i\sum _{j=h}^l t_j}(\rho ^\prime_h)=
\Delta^{-i \sum _{j=h}^l t_j+\beta_h}(\rho _h).$$

Now we replace the last term by $
u_h^{-i\sum _{j=h}^l t_j+\beta_h}.
$ 

We define $$
F_{m_0, m_1,m_2, \ldots, m_l}(u_1, u_2, \ldots, u_l)=$$
$$2\pi \int 
\overset{l}{\underset{h=1} {\prod}} u_h^{-i\sum _{j=h}^lt_j+\beta_h}
P_{m_0,\alpha_0}(-\overset{l}{\underset{j=1} {\sum}}t_j)
\overset{l}{\underset{j=1} {\prod}} G_{m_j,\alpha_j}(t_j)dt_j.
$$
Moreover, we can write 
$$
2\pi P_{m_0,\alpha_0}(-\overset{l}{\underset{j=1} {\sum}}t_j)=\int_{-\infty}^{\infty}e^{s/2} H_{m_0,\alpha_0}(s) e^{-i(\sum _{j=1}^lt_j)(s)} ds.
$$
Using this and assuming that $u_h=e^{s_h}$, we can do the integration. Then the coefficient of $t_j$ in the exponent is $$-is-i\sum _{h=1}^js_h.$$
So integrating in $t_j$ gives the Fourier transform of $G_{m_j, \alpha_j}$ at $s+\sum _{h=1}^js_h$. On the other hand we have
$$
e^{(m_j-\alpha_j)(s+\sum _{h=1}^js_h)}(e^{(s+\sum _{h=1}^js_h)}+1)^{-m_j}=$$
$$
e^{(m_j-\alpha_j)s} \left(  \overset{j}{\underset{h=1} {\prod}} u_h \right)^{(m_j-\alpha_j)}
 \left(  e^s\overset{j}{\underset{h=1} {\prod}} u_h+1 \right)^{-m_j}.
$$
When we multiply these terms from $j=1$ to $j=l$, the exponent of $u_h$ is $\sum _{j=h}^l(m_j-\alpha_j)$. So $u_h^{\beta_h}$ disappears and we get
$$
F_{m_0,m_1,m_2, \dots, m_l}(u_1, u_2, \ldots, u_l)=$$
$$ \int_{-\infty}^{\infty}(e^s+1)^{-m_0}
\overset{l}{\underset{j=1} {\prod}} 
 \left(  e^s\overset{j}{\underset{h=1} {\prod}} u_h+1 \right)^{-m_j}
 e^{(\sum _{j=0}^lm_j-1)s} e^{s/2}ds.
$$
\end{proof}
In Lemma \ref{RE}, it is clear that
$$F_{m_0,m_1,m_2, \ldots, m_l}(u_1, u_2, \ldots, u_l)=H_{m_0,m_1,m_2, \dots, m_l}(u_1,u_1 u_2, \ldots, u_1 \cdots u_l),$$
where
$$H_{m_0,m_1,m_2, \ldots, m_l}(u_1, u_2, \ldots, u_l)=$$
$$\int_{0}^{\infty}(u+1)^{-m_0} \overset{l}{\underset{j=1} {\prod}}\left ( uu_j+1 \right )^{-m_j}u^{(\sum _{j=0}^lm_j-3/2)} du.$$ 
We only need some of these functions:
$$H_{1,1}(x)=\int_{0}^{\infty}(u+1)^{-1} (ux+1)^{-1} u^{1/2} du=
\frac{\pi }{x+\sqrt{x}},$$
$$H_{1,1,1}(x,y)=\int_{0}^{\infty}(u+1)^{-1} (ux+1)^{-1}(uy+1)^{-1} u^{3/2} du=$$
$$
\frac{\pi  \left(\sqrt{x}+\sqrt{y}+1\right)}{\left(\sqrt{x}+1\right) \sqrt{x} \left(y+\sqrt{y}\right) \left(\sqrt{x}+\sqrt{y}\right)},$$
$$H_{2,1}(x)=\int_{0}^{\infty}(u+1)^{-2} (ux+1)^{-1} u^{3/2} du=
\frac{\frac{2 \pi }{\sqrt{x}}+\pi }{2 \left(\sqrt{x}+1\right)^2},$$
$$H_{2,1,1}(x,y)=\int_{0}^{\infty}(u+1)^{-2} (ux+1)^{-1}(uy+1)^{-1} u^{5/2} du=$$
$$
\frac{\pi  \left(\sqrt{x} \left(\sqrt{y}+2\right)^2+x \left(\sqrt{y}+2\right)+2 \left(\sqrt{y}+1\right)^2\right)}{2 \left(\sqrt{x}+1\right)^2 \sqrt{x} \left(\sqrt{y}+1\right)^2 \sqrt{y} \left(\sqrt{x}+\sqrt{y}\right)},
$$
$$H_{1,2,1}(x,y)=\int_{0}^{\infty}(u+1)^{-1} (ux+1)^{-2}(uy+1)^{-1} u^{5/2} du=$$
$$
\frac{\pi  \left(2 x^{3/2}+4 x \left(\sqrt{y}+1\right)+2 \sqrt{x} \left(\sqrt{y}+1\right)^2+y+\sqrt{y}\right)}{2 \left(\sqrt{x}+1\right)^2 x^{3/2} \left(\sqrt{y}+1\right) \sqrt{y} \left(\sqrt{x}+\sqrt{y}\right)^2},
$$
$$H_{2,2,1}(x,y)=\int_{0}^{\infty}(u+1)^{-2} (ux+1)^{-2}(uy+1)^{-1} u^{7/2} du=
$$
$$
\frac{\pi  \left(2 \left(x^{3/2}+4 x+4 \sqrt{x}+1\right) y+\left(7 x^{3/2}+x^2+13 x+7 \sqrt{x}+1\right) \sqrt{y}\right)}{2 \left(\sqrt{x}+1\right)^3 x^{3/2} \left(\sqrt{y}+1\right)^2 \sqrt{y} \left(\sqrt{x}+\sqrt{y}\right)^2}+
$$
$$
\frac{\pi  \left(\left(x+3 \sqrt{x}+1\right) y^{3/2}+2 \left(\sqrt{x}+1\right)^3 \sqrt{x}\right)}{2 \left(\sqrt{x}+1\right)^3 x^{3/2} \left(\sqrt{y}+1\right)^2 \sqrt{y} \left(\sqrt{x}+\sqrt{y}\right)^2}
$$
$$H_{3,1}(x)=\int_{0}^{\infty}(u+1)^{-3} (ux+1)^{-1} u^{5/2} du
=\frac{\pi  \left(3 x+9 \sqrt{x}+8\right)}{8 \left(\sqrt{x}+1\right)^3 \sqrt{x}},
$$
$$H_{3,1,1}(x,y)=\int_{0}^{\infty}(u+1)^{-3} (ux+1)^{-1}(uy+1)^{-1} u^{7/2} du=
$$
$$
\frac{\pi  \left(-\frac{3 x+9 \sqrt{x}+8}{\left(\sqrt{x}+1\right)^3 \sqrt{x}}-\frac{8}{\sqrt{y}+1}-\frac{5}{\left(\sqrt{y}+1\right)^2}-\frac{2}{\left(\sqrt{y}+1\right)^3}+\frac{8}{\sqrt{y}}\right)}{8 (x-y)}.
$$

Now with the notations that we have set up, we can state and prove the main result of this paper:
\begin{theorem} \label{Main}
The scalar curvature of $A_{\theta}^3$, with the perturbed metric, up to a factor of $-4\sqrt{\pi}/3$ is the element $S\in A_{\theta}^3$, given by
$$
S=k^{2}(-3H_{1,1}
+6H_{2,1}-4H_{3,1})(\Delta_{(1)})\left(\delta _i(k)^2\right)$$
$$
+k^{3}(-3/2H_{1,1}
+7/2H_{2,1}-2H_{3,1})(\Delta_{(1)})(\delta _i\left(\delta _i(k)\right))
$$
$$
+k^{2}(-3H_{1,1}
+5H_{2,1}-4H_{3,1})(\Delta_{(1)})\left(\Delta^{1/6}(\delta _i(k)^2)\right)
$$
$$
+k^{3}(-3/2H_{1,1}
+5/2H_{2,1}-2H_{3,1})(\Delta_{(1)})\left(\Delta^{1/3}(\delta _i(\delta _i(k)))\right)
$$
$$
+k^{3}(-3/2H_{1,1}
+5/2H_{2,1}-2H_{3,1})(\Delta_{(1)})\left(\Delta^{1/6}(\delta _i(\delta _i(k)))\right)
$$
$$
+k^{2}(
4H_{2,1}-4H_{3,1})(\Delta_{(1)})\left(\Delta^{1/3}(\delta _i(k)^2)\right)
$$
$$
+k^{3}(
3/2H_{2,1}-2H_{3,1})(\Delta_{(1)})\left(\Delta^{1/2}(\delta _i(\delta _i(k)))\right)
$$
$$
+k^{2}(-3H_{1,1}
+6H_{2,1}-4H_{3,1})(\Delta_{(1)})\left(\Delta^{1/6}(\delta _i(k))\delta _i(k)\right)
$$
$$
+k^{2}(5H_{2,1}-4H_{3,1})(\Delta_{(1)})\left(\Delta^{1/3}(\delta _i(k))\delta _i(k)\right)
$$
$$
+k^{2}(5/2H_{1,1,1}-8H_{2,1,1}+4H_{3,1,1})
(\Delta_{(1)},\Delta_{(1)}\Delta_{(2)})
(\Delta^{5/6}(\delta _i(k)) \delta _i(k))
$$
$$
+k^{2}(-H_{1,2,1}+2H_{2,2,1})
(\Delta_{(1)},\Delta_{(1)}\Delta_{(2)})
(\Delta^{3/2}(\delta _i(k)) \delta _i(k))
$$
$$
+k^{2}(5/2H_{1,1,1}-8H_{2,1,1}+4H_{3,1,1})
(\Delta_{(1)},\Delta_{(1)}\Delta_{(2)})
(\Delta^{2/3}(\delta _i(k)) \delta _i(k))
$$
$$
+k^{2}(-H_{1,2,1}+2H_{2,2,1})
(\Delta_{(1)},\Delta_{(1)}\Delta_{(2)})
(\Delta^{4/3}(\delta _i(k)) \delta _i(k))
$$
$$
+k^{2}(5H_{1,1,1}-9H_{2,1,1}+4H_{3,1,1})
(\Delta_{(1)},\Delta_{(1)}\Delta_{(2)})
(\Delta^{1/2}(\delta _i(k)) \delta _i(k))
$$
$$
+k^{2}(3H_{1,1,1}-7H_{2,1,1}+4H_{3,1,1})
(\Delta_{(1)},\Delta_{(1)}\Delta_{(2)})
(\Delta^{1/2}(\delta _i(k)) \Delta^{1/2}(\delta _i(k)))
$$
$$
+k^{2}(-2H_{1,2,1}+2H_{2,2,1})
(\Delta_{(1)},\Delta_{(1)}\Delta_{(2)})
(\Delta^{7/6}(\delta _i(k)) \delta _i(k))
$$
$$
+k^{2}(4H_{2,1}-4H_{3,1})
(\Delta_{(1)})
(\Delta^{1/3}(\delta _i(k)) \Delta^{1/6}(\delta _i(k)))
$$
$$
+k^{2}(-7H_{2,1,1}+4H_{3,1,1})
(\Delta_{(1)},\Delta_{(1)}\Delta_{(2)})
(\Delta^{}(\delta _i(k)) \delta _i(k))
$$
$$
+2k^{2}H_{2,2,1}
(\Delta_{(1)},\Delta_{(1)}\Delta_{(2)})
(\Delta^{5/3}(\delta _i(k)) \delta _i(k))
$$
$$
+k^{2}(3/2H_{1,1,1}-6H_{2,1,1}+4H_{3,1,1})
(\Delta_{(1)},\Delta_{(1)}\Delta_{(2)})
(\Delta^{5/6}(\delta _i(k)) \Delta^{1/2}(\delta _i(k)))
$$
$$
+k^{2}(2H_{1,1,1}-7H_{2,1,1}+4H_{3,1,1})
(\Delta_{(1)},\Delta_{(1)}\Delta_{(2)})
(\Delta^{5/6}(\delta _i(k)) \Delta^{1/3}(\delta _i(k)))
$$
$$
+k^{2}(2H_{1,1,1}-7H_{2,1,1}+4H_{3,1,1})
(\Delta_{(1)},\Delta_{(1)}\Delta_{(2)})
(\Delta^{5/6}(\delta _i(k)) \Delta^{1/6}(\delta _i(k)))
$$
$$
+k^{2}(-H_{1,2,1}+2H_{2,2,1})
(\Delta_{(1)},\Delta_{(1)}\Delta_{(2)})
(\Delta^{3/2}(\delta _i(k)) \Delta^{1/2}(\delta _i(k)))
$$
$$
+k^{2}(-H_{1,2,1}+2H_{2,2,1})
(\Delta_{(1)},\Delta_{(1)}\Delta_{(2)})
(\Delta^{3/2}(\delta _i(k)) \Delta^{1/2}(\delta _i(k)))
$$
$$
+k^{2}(-H_{1,2,1}+2H_{2,2,1})
(\Delta_{(1)},\Delta_{(1)}\Delta_{(2)})
(\Delta^{3/2}(\delta _i(k)) \Delta^{1/6}(\delta _i(k)))
$$
$$
+k^{2}(3/2H_{1,1,1}-6H_{2,1,1}+4H_{3,1,1})
(\Delta_{(1)},\Delta_{(1)}\Delta_{(2)})
(\Delta^{2/3}(\delta _i(k)) \Delta^{1/2}(\delta _i(k)))
$$
$$
+k^{2}(2H_{1,1,1}-7H_{2,1,1}+4H_{3,1,1})
(\Delta_{(1)},\Delta_{(1)}\Delta_{(2)})
(\Delta^{2/3}(\delta _i(k)) \Delta^{1/3}(\delta _i(k)))
$$
$$
+k^{2}(2H_{1,1,1}-7H_{2,1,1}+4H_{3,1,1})
(\Delta_{(1)},\Delta_{(1)}\Delta_{(2)})
(\Delta^{2/3}(\delta _i(k)) \Delta^{1/6}(\delta _i(k)))
$$
$$
+k^{2}(-H_{1,2,1}+2H_{2,2,1})
(\Delta_{(1)},\Delta_{(1)}\Delta_{(2)})
(\Delta^{4/3}(\delta _i(k)) \Delta^{1/2}(\delta _i(k)))
$$
$$
+k^{2}(-H_{1,2,1}+2H_{2,2,1})
(\Delta_{(1)},\Delta_{(1)}\Delta_{(2)})
(\Delta^{4/3}(\delta _i(k)) \Delta^{1/3}(\delta _i(k)))
$$
$$
+k^{2}(-H_{1,2,1}+2H_{2,2,1})
(\Delta_{(1)},\Delta_{(1)}\Delta_{(2)})
(\Delta^{4/3}(\delta _i(k)) \Delta^{1/6}(\delta _i(k)))
$$
$$
+k^{2}(4H_{1,1,1}-8H_{2,1,1}+4H_{3,1,1})
(\Delta_{(1)},\Delta_{(1)}\Delta_{(2)})
(\Delta^{1/2}(\delta _i(k)) \Delta^{1/3}(\delta _i(k)))
$$
$$
+k^{2}(4H_{1,1,1}-8H_{2,1,1}+4H_{3,1,1})
(\Delta_{(1)},\Delta_{(1)}\Delta_{(2)})
(\Delta^{1/2}(\delta _i(k)) \Delta^{1/6}(\delta _i(k)))
$$
$$
+k^{2}(-2H_{1,2,1}+2H_{2,2,1})
(\Delta_{(1)},\Delta_{(1)}\Delta_{(2)})
(\Delta^{7/6}(\delta _i(k)) \Delta^{1/2}(\delta _i(k)))
$$
$$
+k^{2}(-2H_{1,2,1}+2H_{2,2,1})
(\Delta_{(1)},\Delta_{(1)}\Delta_{(2)})
(\Delta^{7/6}(\delta _i(k)) \Delta^{1/3}(\delta _i(k)))
$$
$$
+k^{2}(-2H_{1,2,1}+2H_{2,2,1})
(\Delta_{(1)},\Delta_{(1)}\Delta_{(2)})
(\Delta^{7/6}(\delta _i(k)) \Delta^{1/6}(\delta _i(k)))
$$
$$
+k^{2}(-5H_{2,1,1}+4H_{3,1,1})
(\Delta_{(1)},\Delta_{(1)}\Delta_{(2)})
(\Delta^{}(\delta _i(k)) \Delta^{1/2}(\delta _i(k)))
$$
$$
+k^{2}(-6H_{2,1,1}+4H_{3,1,1})
(\Delta_{(1)},\Delta_{(1)}\Delta_{(2)})
(\Delta^{}(\delta _i(k)) \Delta^{1/3}(\delta _i(k)))
$$
$$
+k^{2}(-6H_{2,1,1}+4H_{3,1,1})
(\Delta_{(1)},\Delta_{(1)}\Delta_{(2)})
(\Delta^{}(\delta _i(k)) \Delta^{1/6}(\delta _i(k)))
$$
$$
+2k^{2}H_{2,2,1}
(\Delta_{(1)},\Delta_{(1)}\Delta_{(2)})
(\Delta^{5/3}(\delta _i(k)) \Delta^{1/2}(\delta _i(k)))
$$
$$
+2k^{2}H_{2,2,1}
(\Delta_{(1)},\Delta_{(1)}\Delta_{(2)})
(\Delta^{5/3}(\delta _i(k)) \Delta^{1/3}(\delta _i(k)))
$$
$$
+2k^{2}H_{2,2,1}
(\Delta_{(1)},\Delta_{(1)}\Delta_{(2)})
(\Delta^{5/3}(\delta _i(k)) \Delta^{1/6}(\delta _i(k)))
$$
$$
-3/2k^{4}
H_{1,1}(\Delta_{(1)})\left(\Delta^{1/3}(\delta _i(k^{-1}))\delta _i(k)\right)
$$
$$
-3/2k^{4}
H_{1,1}(\Delta_{(1)})\left(\Delta^{1/3}(\delta _i(k^{-1}))\Delta^{1/3}(\delta _i(k))\right)
$$
$$
-3/2k^{4}
H_{1,1}(\Delta_{(1)})\left(\Delta^{1/6}(\delta _i(k^{-1}))\delta _i(k)\right)
$$
$$
-3/2k^{4}
H_{1,1}(\Delta_{(1)})\left(\Delta^{1/6}(\delta _i(k^{-1}))\Delta^{1/3}(\delta _i(k))\right)
$$
$$
-3/2k^{4}
H_{1,1}(\Delta_{(1)})\left(\Delta^{1/3}(\delta _i(k^{-1}))\Delta^{1/6}(\delta _i(k))\right)
$$
$$
-3/2k^{4}
H_{1,1}(\Delta_{(1)})\left(\Delta^{1/6}(\delta _i(k^{-1}))\Delta^{1/6}(\delta _i(k))\right)
$$
\end{theorem}
\begin{proof}
It suffices to find 
$$\int_{0}^{\infty} B(r) dr.$$ 
For that we only need to use the substitution $r^2=u$, and then apply Lemma \ref{RE}. 
\end{proof}
Now we shall show that the formula in Theorem \ref{Main} is compatible with the 
formula in the commutative case. In fact, in the commutative case the modular operator is the identity operator. So it suffices to find the limit of $S$ in Theorem \ref{Main}. After simplification we see that 
$$
S=
\lim_{x \to 1} k^{3}  (-9/2H_{1,1}
+10H_{2,1}-8H_{3,1}
)(x)\left(\delta _i(\delta _i(k))\right)
$$
$$+\lim_{(x,y) \to (1,1)}k^2(-9H_{1,1}+30H_{2,1}-24
H_{3,1}+ 32H_{1,1,1}-112 H_{2,1,1}$$
$$+64H_{3,1,1}-16H_{1,2,1}+
32H_{2,2,1})(x,y)(\delta _i(k)\delta _i(k))
$$
$$
-\lim_{x \to 1}9k^{4}
H_{1,1}(x)\left(\delta _i(k^{-1})\delta _i(k)\right).
$$
Using the derivation property in the last term,  we get
$$
S=
\lim_{x \to 1}k^{3}(-9/2H_{1,1}
+10H_{2,1}-8H_{3,1}
)(x)\left(\delta _i(\delta _i(k))\right)
$$
$$+\lim_{(x,y) \to (1,1)}k^2(-9H_{1,1}+30H_{2,1}-24
H_{3,1}+ 32H_{1,1,1}-112 H_{2,1,1}$$
$$+64H_{3,1,1}-16H_{1,2,1}+
32H_{2,2,1})(x,y)(\delta _i(k)\delta _i(k))
$$
$$
+\lim_{x \to 1}9k^{2}
H_{1,1}(x)(\delta _i(k)\delta _i(k)).
$$
Therefore,
$$
S=
k^{3}(-9\pi/4
+15\pi/4-5\pi/2
)\left(\delta _i(\delta _i(k))\right)
$$
$$+k^2(-9\pi/2+45\pi/4-30\pi/4+ 12\pi-35\pi$$
$$+35\pi/2-5\pi+
35\pi/4+9\pi/2)(\delta _i(k)\delta _i(k)).
$$
So
\begin{equation}\label{wt}
S=-\pi k^{3}\left(\delta _i(\delta _i(k))\right)
+2 \pi k^2(\delta _i(k)\delta _i(k)).
\end{equation}

On the other hand, by applying Lemma 5.1 in \cite{kh2}, we obtain 
\begin{equation}\label{wr}
k^{-2}\delta _i(k)\delta _i(k)=\delta _i(\log k)\delta _i(\log k),
\end{equation}
and
\begin{equation}\label{we}
k^{-1}\delta _i(\delta _i(k))=\delta _i(\delta _i(\log k))+\delta _i(\log k)\delta _i(\log k).
\end{equation}
Using (\ref{wr}) and (\ref{we}) in (\ref{wt}), we see that
$$
S=- \pi k^4 \delta _i(\delta _i(\log k)) + \pi k^4\delta _i(\log k)\delta _i(\log k),
$$
which is the same formula as in the classic case up to a normalization factor.



\begin{thebibliography}{3.30}
\bibitem[1]{am} J. Arnlind, M. Wilson, \textit{Riemannian curvature of the noncommutative 3-sphere}, J. Noncommut. Geom., (to appear), arXiv:1604.01159v1.

\bibitem [2] {bhumar} T. A. Bhuyain and M. Marcolli, {\it The Ricci flow on noncommutative two-tori}.
Lett. Math. Phys. 101, no. 2, 173-194, (2012).
\bibitem [3]{cohen} P. B. Cohen, A. Connes, \textit{Conformal geometry of the irrational rotation algebra}. Preprint
MPI / 92-93.
\bibitem [4]{con} A.~Connes,  {\it $C^*$-alg\`ebres et g\'eom\'etrie diff\'erentielle}. C.R. Acad. Sc. Paris,
t.~290, S\'erie A, 599-604, (1980).

\bibitem [5]{con1}A. Connes, {\it Noncommutative geometry}.  Academic Press (1994).

\bibitem [6]{CM}
A. Connes and H. Moscovici, \textit{Modular curvature for noncommutative two-tori}. J. Amer.
Math. Soc., 27(3):639-684, (2014).
\bibitem [7]{Conn}
A. Connes, P. Tretkoff, \textit{The Gauss-Bonnet theorem for the noncommutative two torus}, \textit{Noncommutative
geometry, arithmetic, and related topics}, 141-158, Johns Hopkins Univ. Press,
Baltimore, MD, (2011).
\bibitem [8]{das}L. Dabrowski, A. Sitarz, \textit{Asymmetric noncommutative torus},  SIGMA Symmetry Integrability Geom. Methods Appl. 11, Paper 075, (2015).
\bibitem [9]{sw}M. Eckstein, A. Sitarz, R. Wulkenhaar, \textit{The Moyal Sphere}, arXiv:1601.05576.
\bibitem[10]{fk} A. Fathi, M. Khalkhali, \textit{On Certain Spectral Invariants of Dirac Operators on Noncommutative Tori}, arXiv:1504.01174v1. 
\bibitem[11]{fka} A. Fathi, A. Gorbanpour, M. Khalkhali, \textit{The Curvature of the Determinant Line Bundle on the Noncommutative Two Torus}, arXiv:1410.0475v1.
\bibitem [12]{farrr} F. Fathizadeh, O. Gabriel, \textit{On the Chern-Gauss-Bonnet theorem and conformally twisted spectral triples for C*-dynamical systems.} SIGMA Symmetry Integrability Geom. Methods Appl. 12, Paper 016, (2016). 
\bibitem [13] {far1}F. Fathizadeh, \textit{On the Scalar Curvature for the Noncommutative Four Torus}, J. Math. Phys. 56, no. 6, 062303, (2015).
\bibitem [14]{kh}
F. Fathizadeh, M. Khalkhali, \textit{The Gauss-Bonnet theorem for noncommutative two tori with
a general conformal structure}, J. Noncommut. Geom., 6, no. 3, 457-480, (2012).
\bibitem [15]{kh2}
F. Fathizadeh and M. Khalkhali, \textit{Scalar curvature for the noncommutative two torus. J.
Noncommut}. Geom., 7(4):1145-1183, (2013).
\bibitem [16]{kh3}
F. Fathizadeh and M. Khalkhali.
Scalar curvature for noncommutative four-tori, J. Noncommut. Geom., 9(2), 473-503, (2015).
\bibitem [17]{kh4}
F. Fathizadeh, M Khalkhali,  \textit{Weyl's law and Connes' trace theorem for noncommutative two tori}. Lett. Math. Phys. 103, no. 1, 1-18, (2013).
\bibitem [18] {gil}
P. Gilkey, \textit{Invariance theory, the heat equation, and the Atiyah-Singer index theorem}, Mathematics
Lecture Series, 11. Publish or Perish, Inc., Wilmington, DE, (1984).
 \bibitem[19]{km} M. Khalkhali, A. Moatadelro, \textit{A Riemann–Roch theorem for the noncommutative two torus.} Journal of Geometry and Physics 86, 19-30 (2014).
 \bibitem [20]{lesch} M. Lesch. \textit{Divided differences in noncommutative geometry: Rearrangement lemma, functional calculus and expansional formula,} J. Noncommut. Geom., (to appear), arXiv:1405.0863v2.
 \bibitem [21] {ml} M. Lesch, H. Moscovici, \textit{Modular Curvature and Morita Equivalence.} Geom. Funct. Anal. 26, no. 3, 818-873, (2016).

\bibitem [22]{Liu} Y. Liu, \textit{Modular curvature for toric noncommutative manifolds}, arXiv:1510.04668.
\bibitem[23]{ros} J. Rosenberg, \textit{Levi-Civita's theorem for noncommutative tori}, SIGMA Symmetry Integrability Geom. Methods Appl. 9, Paper 071, (2013).
\end{thebibliography}
\end{document}